\documentclass[11pt]{amsart}

\usepackage{amssymb}
\usepackage{amsmath}
\usepackage{amsthm}
\usepackage{xcolor}

\theoremstyle{plain}

\newtheorem{theorem}{Theorem}[section]

\newtheorem{lemma}[theorem]{Lemma}
\newtheorem{proposition}[theorem]{Proposition}
\newtheorem{BA}[theorem]{Basic assumptions}
\newtheorem{corollary}[theorem]{Corollary}

\newtheorem{fact}[theorem]{Fact}

\theoremstyle{definition}

\newtheorem{definition}[theorem]{Definition}
\newtheorem{remark}[theorem]{Remark}
\newtheorem{example}[theorem]{Example}

\newtheorem*{problem*}{Problem}

\newtheorem{problem}[theorem]{Problem}

\newcommand{\dist}{\mathrm{dist}}
\newcommand{\C}{\mathrm{c}}

\newcommand{\R}{\mathbb{R}}
\newcommand{\N}{\mathbb{N}}

\newcommand{\inte}{\mathrm{int}\,}
\newcommand{\cconv}{\overline{\mathrm{conv}}\,}

\newcommand{\conv}{{\mathrm{conv}}\,}

\renewcommand{\epsilon}{\varepsilon}
\renewcommand{\phi}{\varphi}

\begin{document}

\title{Regularity and stability for a convex feasibility problem}

\author{Carlo Alberto De Bernardi
}
\address{Dipartimento di Matematica per le Scienze economiche, finanziarie ed attuariali, Universit\`{a} Cattolica del Sacro Cuore, Via Necchi 9, 20123 Milano, Italy}

\email{carloalberto.debernardi@unicatt.it, carloalberto.debernardi@gmail.com}

\author{Enrico Miglierina
}
\address{Dipartimento di Matematica per le Scienze economiche, finanziarie ed attuariali, Universit\`{a} Cattolica del Sacro Cuore, Via Necchi 9, 20123 Milano, Italy}

\email{enrico.miglierina@unicatt.it}



 \thanks{
}

%%%%%%%%%%%%%%%%%%%%%%%%%%%%%%%%%%%%%%%%%%%%%%%%%%%%%%

%%%%%%%%%%%%%%%%%% Headings
%\markboth{authors}{}
%%%%%%%%%%%%%%%%%%%%%%%%%%%

% REQUIRED
\begin{abstract}
	Let us consider two sequences of closed convex  sets $\{A_n\}$ and $\{B_n\}$  converging with respect to the Attouch-Wets convergence to $A$ and $B$, respectively. Given a starting point $a_0$, we consider the sequences of points obtained by projecting on the ``perturbed'' sets, i.e., the sequences $\{a_n\}$ and $\{b_n\}$ defined inductively by $b_n=P_{B_n}(a_{n-1})$ and $a_n=P_{A_n}(b_n)$.  	
	Suppose that $A\cap B$ (or a suitable substitute if $A \cap B=\emptyset$) is bounded, we prove that if the couple $(A,B)$ is (boundedly) regular then the couple $(A,B)$ is $d$-stable, i.e., for each $\{a_n\}$ and $\{b_n\}$ as above we have   $\mathrm{dist}(a_n,A\cap B)\to 0$ and $\mathrm{dist}(b_n,A\cap B)\to 0$. 
\end{abstract}

\maketitle

\section{Introduction}
Let $A$ and $B$ be two closed 
convex nonempty sets  in a Hilbert space $X$. The (2-set) convex feasibility problem asks to find a point in the intersection of $A$ and $B$ (or, when $A \cap B=\emptyset$, a pair of points, one in $A$ and the other in $B$, that realizes the distance between $A$ and $B$). The relevance of this problem is  due to the fact that many mathematical and concrete problems in applications can be formulated as a convex feasibility
problem. As typical examples, we mention solution of convex inequalities, partial differential equations, minimization of convex nonsmooth functions, medical imaging, computerized tomography and image reconstruction. 

The method of alternating projections is the simplest iterative procedure for finding a solution of the convex feasibility problem and it goes back to von Neumann \cite{vonNeumann}: let us denote by $P_A$ and $P_B$ the projections on the sets $A$ and $B$, respectively, and, given a starting point $c_0\in X$, consider the {\em alternating projections sequences} $\{c_n\}$ and $\{d_n\}$ given  by $$d_n=P_{B}(c_{n-1})\ \ \text{and}\ \ c_n=P_{A}(d_n)\ \ \ \ \ (n\in\N).$$
If the sequences $\{c_n\}$ and $\{d_n\}$ converge in norm, we say that the method of alternating projections converges.
Originally, von Neumann proved that the method of alternating projection converges when $A$ and $B$ are closed subspace. Then, for two generic convex sets, the weak convergence of the alternating projection sequences was proved by Bregman in 1965 (\cite{Bregman}). Nevertheless, the problem whether the alternating projections algorithm converges in norm for each couple of convex sets remained open till the example given by Hundal in 2004 (\cite{Hundal}). This example shows that the alternating projections do not converge in norm when $A$ is a suitable convex cone and $B$ is a hyperplane touching the vertex of $A$. Moreover, this example emphasizes the importance of finding sufficient conditions ensuring the norm convergence of the alternating projections algorithm. In the literature, conditions of this type were studied (see, e.g., \cite{BauschkeBorwein93,BorweinSimsTam}), even before the example by Hundal.
Here, we focus on those conditions based on the notions of regularity, introduced in \cite{BauschkeBorwein93}. Indeed, in the present paper, we investigate the relationships between regularity of the couple $(A,B)$ (see Definition~\ref{def: regularity} below) and ``stability'' properties of the alternating projections method in the following sense. Let us suppose that 
$\{A_n\}$ and $\{B_n\}$
are two sequences of closed convex sets such that $A_n\rightarrow
A$ and $B_n\rightarrow B$ for the Attouch-Wets {variational} convergence (see Definition~\ref{def:AW}) and let us introduce the definition of {\em perturbed alternating projections sequences}.

\begin{definition}\label{def:perturbedseq} Given $a_0\in X$, the {\em perturbed alternating projections sequences}  $\{a_n\}$ and $\{b_n\}$, w.r.t. $\{A_n\}$ and $\{B_n\}$ and with starting point $a_0$, are defined inductively by
	$$b_n=P_{B_n}(a_{n-1})\ \ \ \text{and}\ \ \  a_n=P_{A_n}(b_n) \ \ \ \ \ \ \ \ \ (n\in\N).$$ 
\end{definition}

\noindent Our aim is to find some conditions on the limit sets $A$ and $B$ such that, for each choice of the sequences $\{A_n\}$ and $\{B_n\}$ and for each choice of  the starting point $a_0$, 
the corresponding perturbed alternating projections sequences  $\{a_n\}$ and $\{b_n\}$ satisfy $\mathrm{dist}(a_n,A\cap B)\to 0$ and $\mathrm{dist}(b_n,A\cap B)\to 0$. If this is the case, we say that the couple $(A,B)$ is {\em $d$-stable}.
In particular, we show that the  regularity of the couple $(A,B)$  implies not only the norm convergence of the alternating projections sequences for the couple $(A,B)$ (as already known from \cite{BauschkeBorwein93}), but also that the couple $(A,B)$ is $d$-stable. This result 
%is a theoretical issue that links regularity and "stability" properties of a convex feasibility theorem, but it
  might be interesting also in  applications since real data are often affected by some uncertainties. Hence stability of the convex feasibility problem with respect to data perturbations is a desirable property, also in view of computational developments. 

Let us conclude the introduction by a brief description of the structure of the paper. In Section \ref{SEction notations}, we list some notations and definitions, and we recall some well-known facts about the alternating projections method. Section~\ref{Section regularity} is devoted to various notions of regularity and their relationships. It is worth pointing out that in this section we provide a new and alternative proof of the convergence of the alternating projections algorithm under regularity assumptions. This proof well illustrates the main geometrical idea behind the proof of our main result Theorem~\ref{teo: mainHilbert}, stated and proved in Section \ref{Section Main Result}. This result shows that {\em a regular couple $(A,B)$ is $d$-stable whenever $A\cap B$ (or a suitable substitute if $A \cap B=\emptyset$) is bounded}. 
Corollaries~\ref{Corollary:stronglyexp}, \ref{corollary:corpilur}, and \ref{cor:sottospazisommachiusa} simplify and generalize some of the results  obtained in \cite{DebeMigl}, since there we considered only the case where $A \cap B\neq\emptyset$ whereas, in the present paper, we encompass also the situation where the intersection of $A$ and $B$ is empty. We conclude the paper with Section 5, where we discuss the necessity of the assumptions of our main result and we state a natural open problem: suppose that $A \cap B$ is bounded, is regularity equivalent to $d$-stability?

\section{Notation and preliminaries }\label{SEction notations}

Throughout all this paper, $X$ denotes a nontrivial real normed space with
the topological dual $X^*$. We
denote by $B_X$ and $S_X$ the closed unit ball and the unit sphere of $X$, respectively. 
For $x,y\in X$, $[x,y]$ denotes the closed segment in $X$ with
endpoints $x$ and $y$, and $(x,y)=[x,y]\setminus\{x,y\}$ is the
corresponding ``open'' segment. 
%For a subset $K$ of $X$,
%$\alpha>0$, and a functional $x^*\in \SX$ bounded on $K$, let
%$$S(x^*,\alpha,K)=\{x\in
%K;\, x^*(x)\geq\sup x^*(K)-\alpha\}$$ be the closed slice of $K$
%given by $\alpha$ and $x^*$.
For a subset $A$ of $X$, we denote by $\inte(A)$, $\conv(A)$ and
$\cconv(A)$ the interior, the convex hull and the closed convex
hull of $A$, respectively.
Let us recall that a body is a closed convex setsin $X$ with nonempty interior.

 We denote by $$\textstyle \mathrm{diam}(A)=\sup_{x,y\in A}\|x-y\|,$$
the (possibly infinite) diameter of $A$. For $x\in X$, let
$$\dist(x,A) =\inf_{a\in A} \|a-x\|.$$ Moreover, given $A,B$
nonempty subsets of $X$,  we denote by $\dist(A,B)$ the usual
``distance'' between $A$ and $B$, that is,
$$ \dist(A,B)=\inf_{a\in A} \dist(a,B).$$

Now, we recall two notions of convergence for sequences of sets (for a wide overview about this topic see, e.g., \cite{Beer}). By $\C(X)$ we denote the family of all nonempty closed subsets of
	$X$.
Let us introduce the (extended) Hausdorff metric $h$ on
$\C(X)$. For $A,B\in\C(X)$, we define the excess of $A$ over $B$
as
$$e(A,B) = \sup_{a\in A} \mathrm{dist}(a,B).$$
\noindent Moreover, if $A\neq\emptyset$ and $B=\emptyset$ we put
$e(A,B)=\infty$,  if $A=\emptyset$ we put $e(A,B)=0$. Then, we define

$$h(A,B)=\max \bigl\{ e(A,B),e(B,A) \bigr\}.$$

\begin{definition} A sequence $\{A_j\}$ in $\C(X)$ is said to
	Hausdorff converge to $A\in\C(X)$ if $$\textstyle \lim_j h(A_j,A)
	= 0.$$
\end{definition}

As the second notion of convergence, we consider the so called Attouch-Wets convergence (see,
e.g., \cite[Definition~8.2.13]{LUCC}), which can be seen as a
localization of the Hausdorff convergence.  If $N\in\N$ and
$A,B\in\C(X)$, define
\begin{eqnarray*}
	e_N(A,B) &=& e(A\cap N B_X, B)\in[0,\infty),\\
	h_N(A,B) &=& \max\{e_N(A,B), e_N(B,A)\}.
\end{eqnarray*}

\begin{definition}\label{def:AW} A sequence $\{A_j\}$ in $\C(X)$ is said to
	Attouch-Wets converge to $A\in\C(X)$ if, for each $N\in\N$,
	$$\textstyle \lim_j h_N(A_j,A)= 0.$$
\end{definition}

We conclude this section by recalling some well known results about distance between two convex sets and projection of a point onto a convex set.

Suppose that $A$ and $B$ are closed convex nonempty subsets of $X$, we denote
\begin{eqnarray*}
	E&=&\{a\in A;\, d(a,B)=d(A,B)\},\\
	F&=&\{b\in B;\, d(b,B)=d(A,B)\}.
\end{eqnarray*}

{If $C$ is a nonempty closed convex subset of $X$, the projection of a point $x$ onto $C$ is denoted by $P_Cx$.}  
We say that $v=P_{\overline{B-A}}(0)$ is the {\em displacement vector} for the couple $(A,B)$.
It is clear that if $A\cap B\neq\emptyset$ then $E=F=A\cap B$ and the displacement vector for the couple $(A,B)$ is null.
We {recall} the following fact, where, given a map $T:X\rightarrow X$, $\mathrm{Fix}(T)$ denotes the set of fixed points of $T$.

\begin{fact}[{\cite[Fact~1.1]{BauschkeBorwein93}}]\label{fact: BB93} Suppose that $X$ is a Hilbert space and that $A,B$ are closed convex nonempty subsets of $X$. Then we have:
	\begin{enumerate}
		\item $\|v\|=d(A,B)$ and $E+v=F$;
		\item $E=\mathrm{Fix}(P_A P_B)=A\cap(B-v)$ and $F=\mathrm{Fix}(P_B P_A)=B\cap(A+v)$;
		\item $P_B e=P_F e=e+v$ ($e\in E$) and $P_A f=P_E f=f-v$ ($f\in F$).
 	\end{enumerate}
	\end{fact}

\section{Notions of regularity for a couple of convex sets} \label{Section regularity}
{
In this section we introduce some notions of regularity for a couple of nonempty closed convex sets $A$ and $B$. This class of notions was originally introduced in \cite{BauschkeBorwein93}, in order to obtain some conditions ensuring the norm convergence of the alternating projections algorithm (see, also, \cite{BorweinZhu}).
Here we list three different type of regularity: (i) and (ii) are exactly as they appeared in \cite{BauschkeBorwein93}, whereas (iii) is new.
}

\begin{definition}\label{def: regularity} Let $X$ be a Hilbert space and  $A,B$  closed convex nonempty subsets of $X$. Suppose that $E,F$ are nonempty. We say that the couple $(A,B)$ is:\begin{enumerate}
		\item {\em regular} if for each $\epsilon>0$ there exists $\delta>0$ such that $\mathrm{dist}(x,E)\leq \epsilon$, whenever $x\in X$ satisfies
		$$\max\{\mathrm{dist}(x,A),\mathrm{dist}(x,B-v)\}\leq\delta;$$
\item {\em boundedly regular} if for each bounded set $S\subset X$ and for each $\epsilon>0$ there exists $\delta>0$ such that $\mathrm{dist}(x,E)\leq \epsilon$, whenever $x\in S$ satisfies
$$\max\{\mathrm{dist}(x,A),\mathrm{dist}(x,B-v)\}\leq\delta;$$
\item {\em linearly regular for points bounded away from $E$} if for each $\epsilon>0$ there exists $K>0$ such that 
$$\mathrm{dist}(x,E)\leq K\max\{\mathrm{dist}(x,A),\mathrm{dist}(x,B-v)\},$$
whenever  $\mathrm{dist}(x,E)\geq \epsilon$.
	\end{enumerate}  
\end{definition}

The following proposition shows that (i) and (iii) in the definition above are equivalent. The latter part of the proposition is a generalization of \cite[Theorem~3.15]{BauschkeBorwein93}.

\begin{proposition}\label{prop: regular-largedistances} Let $X$ be a Hilbert space and  $A,B$  closed convex nonempty subsets of $X$. Suppose that $E, F$ are nonempty. Let us consider the following conditions. \begin{enumerate}
		\item The couple $(A,B)$ is regular.
		\item  The couple $(A,B)$ is boundedly regular.
		\item The couple $(A,B)$
		is linearly regular for points bounded away from $E$.
	\end{enumerate}
Then $(iii)\Leftrightarrow(i)\Rightarrow(ii)$. Moreover, if $E$ is bounded, then $(ii)\Rightarrow(i)$.
\end{proposition}

\begin{proof}
The implications $(iii)\Rightarrow(i)\Rightarrow(ii)$ are trivial. Let us prove that $(i)\Rightarrow(iii)$. Suppose on the contrary that there exists $\epsilon>0$ and a sequence $\{x_n\}\subset X$ such that $\mathrm{dist}(x_n,E)>\epsilon$  ($n\in\N$) and  
$$\textstyle \frac{\max\{\mathrm{dist}(x_n,A),\mathrm{dist}(x_n,B-v)\}}{\mathrm{dist}(x_n,E)}\to 0.$$
For each $n\in\N$, let $e_n\in E$, $a_n\in A$, and $b_n\in B$ be such that $\|e_n-x_n\|=\mathrm{dist}(x_n,E)$, $\|a_n-x_n\|=\mathrm{dist}(x_n,A)$, and $\|b_n-v-x_n\|=\mathrm{dist}(x_n,B-v)$. Put $\lambda_n=\frac{\epsilon}{\|e_n-x_n\|}\in(0,1)$ and define $z_n=\lambda_n x_n+(1-\lambda_n)e_n$, 
$a'_n=\lambda_n a_n+(1-\lambda_n)e_n\in A$, and $b'_n=\lambda_n b_n+(1-\lambda_n)(e_n+v)\in B$.
By our construction, it is clear that
$$\textstyle \frac{\mathrm{dist}(z_n,A)}{\epsilon}\leq\frac{\|z_n-a'_n\|}{\epsilon}=\frac{\|x_n-a_n\|}{\|e_n-x_n\|}\ \ \  \text{and}\ \ \  \frac{\mathrm{dist}(z_n,B-v)}{\epsilon}\leq\frac{\|b'_n-v-z_n\|}{\epsilon}=\frac{\|b_n-v-x_n\|}{\|e_n-x_n\|}.$$
Hence, $\mathrm{dist}(z_n,E)=\epsilon$ and $\max\{\mathrm{dist}(z_n,A),\mathrm{dist}(z_n,B-v)\}\to 0$. This contradicts (i) and the proof is concluded.

Now, suppose that $E$ is bounded, and let us prove that $(ii)\Rightarrow(i)$. Suppose on the contrary that there exists $\epsilon>0$ and a sequence $\{x_n\}\subset X$ such that $\mathrm{dist}(x_n,E)>\epsilon$  ($n\in\N$) and  
$$\textstyle \max\{\mathrm{dist}(x_n,A),\mathrm{dist}(x_n,B-v)\}\to 0.$$
For each $n\in\N$, let $e_n\in E$, $a_n\in A$, and $b_n\in B$ be such that $\|e_n-x_n\|=\mathrm{dist}(x_n,E)$, $\|a_n-x_n\|=\mathrm{dist}(x_n,A)$, and $\|b_n-v-x_n\|=\mathrm{dist}(x_n,B-v)$. Put $\lambda_n=\frac{\epsilon}{\|e_n-x_n\|}\in(0,1)$ and define $z_n=\lambda_n x_n+(1-\lambda_n)e_n$, 
$a'_n=\lambda_n a_n+(1-\lambda_n)e_n\in A$, and $b'_n=\lambda_n b_n+(1-\lambda_n)(e_n+v)\in B$.
By our construction, it is clear that
$$\textstyle {\mathrm{dist}(z_n,A)}\leq{\|z_n-a'_n\|}\leq{\|x_n-a_n\|}$$
and
$$ {\mathrm{dist}(z_n,B-v)}\leq{\|b'_n-v-z_n\|}\leq{\|b_n-v-x_n\|}.$$
Hence, $\mathrm{dist}(z_n,E)=\epsilon$ and $\max\{\mathrm{dist}(z_n,A),\mathrm{dist}(z_n,B-v)\}\to 0$. Moreover, since $E$ is bounded $\{z_n\}$ is a bounded sequence. This contradicts (ii) and the proof is concluded.
\end{proof}

The following theorem follows by \cite[Theorem~3.7]{BauschkeBorwein93}. 

\begin{theorem}
Let $X$ be a Hilbert space and  $A,B$  closed convex nonempty subsets of $X$.  Suppose that the couple $(A,B)$ is regular. Then the alternating projections method converges.
\end{theorem}

We present an alternative proof of this theorem containing a simplified version of the argument that we will use in our main result Theorem~\ref{teo: mainHilbert}. For the sake of simplicity we present a proof in the case $A\cap B$ is nonempty. The proof in the general case is similar.

\begin{proof} 
By \cite[Theorem~3.3, (iv)]{BauschkeBorwein93}, it is sufficient to prove that $\mathrm{dist}(c_n,A\cap B)\to 0$.
	Let us recall that, by the definition of the sequences $\{c_n\}$ and $\{d_n\}$, we have  
	\begin{enumerate}
		\item[($\alpha$)]  $\mathrm{dist}(c_n,A\cap B)\leq\mathrm{dist}(d_n,A\cap B)$ and $\mathrm{dist}(d_{n+1},A\cap B)\leq\mathrm{dist}(c_n,A\cap B).$
	\end{enumerate}
	Let $\epsilon>0$, by the equivalence $(i)\Leftrightarrow(iii)$ in Proposition~\ref{prop: regular-largedistances}, there exists $K>0$ such that $$\mathrm{dist}(x,A\cap B)\leq K\max\{\mathrm{dist}(x,A),\mathrm{dist}(x,B)\},$$
	whenever  $\mathrm{dist}(x,A\cap B)\geq \epsilon$. Observe that {$K\geq 1$} and define $\eta=\sqrt{1-\frac1{K^2}}$.
	{Then a computation, based on some trigonometric considerations in the plane defined by $c_n,d_n$ and $d_{n+1}$,} shows that the following condition holds for each $n\in\N$: 
	\begin{enumerate}
		\item[($\beta$)] if $\mathrm{dist}(c_n,A\cap B)\geq\epsilon$	 then $\textstyle \mathrm{dist}(c_n,A\cap B)\leq\eta\,\mathrm{dist}(d_n,A\cap B);$
		if $\mathrm{dist}(d_{n+1},A\cap B)\geq\epsilon$ then $\textstyle \mathrm{dist}(d_{n+1},A\cap B)\leq\eta\,\mathrm{dist}(c_n,A\cap B).$
	\end{enumerate}
	 {By} taking into account ($\alpha$), ($\beta$), and the fact that $\eta<1$, we have that eventually $\mathrm{dist}(c_n,A\cap B)\leq\epsilon$ and $\mathrm{dist}(d_n,A\cap B)\leq\epsilon$. The proof is concluded.	    
\end{proof}

\section{{Regularity and perturbed alternating projections}}\label{Section Main Result}
{
	This section is devoted to prove our main result.  Indeed, here we show that if a couple $(A,B)$ of convex closed sets is regular then  not only the alternating projections method converges but also the couple $(A,B)$ satisfies certain ``stability'' properties with respect to perturbed projections sequences. In the present section, if not differently stated, $X$ denotes a Hilbert space. If $u,v\in X\setminus\{0\}$, we denote as usual
	$$\textstyle\cos(u,v)=\frac{\langle u,v\rangle}{\|u\|\|v\|},$$
	where $\langle \cdot,\cdot\rangle$ denotes the inner product in $X$. 
	
Let us start by making precise the word   ``stability'' by introducing
} 
the following two notions of  stability for a couple $(A,B)$ of convex closed subsets of $X$.

\begin{definition}\label{def: stability}
	Let  $A$ and $B$ be closed convex subsets of $X$ such that $E, F$ are nonempty. We say that the couple $(A,B)$ is {\em stable} [{\em $d$-stable}, respectively] if for each choice  of sequences $\{A_n\},\{B_n\}\subset\C(X)$ converging {with respect to} the Attouch-Wets convergence to $A$ and $B$, respectively, and for each choice of  the starting point $a_0$, the  corresponding perturbed alternating projections sequences  $\{a_n\}$ and $\{b_n\}$ converge in norm [satisfy 
	$\mathrm{dist}(a_n, E)\to0$ and $\mathrm{dist}(b_n, F)\to0$, respectively].
\end{definition}

\begin{remark}
	We remark that the couple $(A,B)$ is {\em stable} if and only	
	if for each choice  of sequences $\{A_n\},\{B_n\}\subset\C(X)$ converging {with respect to} the Attouch-Wets convergence to $A$ and $B$, respectively,   and for each choice of  the starting point $a_0$, there exists $e\in E$ such that the  perturbed alternating projections sequences  $\{a_n\}$ and $\{b_n\}$ satisfy
 $a_n\to e$ and $b_n\to e+v$ in norm.
\end{remark}

\begin{proof}
		Let us start by proving that if $a_n\to e$ then $e\in E$. It is not difficult to prove that, since $$a_{n+1}= P_{A_n}P_{B_n}a_n=P_{A}P_{B}a_n+(P_{A_n}P_{B_n}-P_{A}P_{B})a_n$$  and since $A_n\to A,B_n\to B$  for the Attouch-Wets convergence, we have $e=P_AP_B e$. By Fact~\ref{fact: BB93}, (ii), we have that $e\in E$. Similarly, it is easy to see that 
	$$b_{n+1}=P_{B_n}a_n=P_{B}a_n+(P_{B_n}-P_{B})a_n\to P_B e=e+v,$$
	and the proof is concluded.
\end{proof}

It is clear that if the couple $(A,B)$ is stable,
then it is $d$-stable. Moreover, if $E,F$ are singletons then also the converse implication holds true.
{The following basic assumptions will be considered in the sequel of the paper.}

\begin{BA}\label{ba}
	Let  $A,B$  be closed convex non\-empty subsets of $X$.  Suppose that:
	\begin{enumerate}
		\item  $E,F$ are nonempty and bounded;
		\item $\{A_n\}$ and $\{B_n\}$ are sequences of closed convex sets such that $A_n\rightarrow
		A$ and $B_n\rightarrow B$ for the Attouch-Wets convergence.
	\end{enumerate}  
\end{BA}

Now, let us prove a chain of lemmas and propositions that we shall use in the proof of our main result, Theorem~\ref{teo: mainHilbert} below.

\begin{lemma}\label{lemma:2dimensional}
	Let $G$ be a closed convex subset of  $X$. Suppose that there exist $\epsilon,K>0$ such that $\epsilon B_X\subset G\subset K B_X$. Then, if $u,w\in\partial G$ and $\cos(u,w)=\theta>0$, we have 
	$$\textstyle \|u-w\|^2\leq K^2(\frac{K^2}{\epsilon^2}+1)\frac{1-\theta^2}{\theta^2}.$$
\end{lemma}

\begin{proof}
	Without any loss of generality we can suppose that $X=\R^2$ and $u=(\|u\|,0)$. Let us denote $w=(x,y)$, with $x,y\in\R$, and suppose that  $u,w\in\partial G$. 
	
	We claim that 
	$\textstyle |y|\geq|\frac{\epsilon}{\|u\|}(x-\|u\|)|.$
	Indeed, since $\epsilon B_X\subset G$, an easy convexity argument shows that
	if $u\in \partial G$ and $w$ is such that $|y|<|\frac{\epsilon}{\|u\|}(x-\|u\|)|$ then we would have $w\notin \partial G $.

%	 
%	Since $\epsilon B_X\subset G$ and $u,w\in\partial G$, an easy convexity argument shows that 
%	$$\textstyle |y|\geq|\frac{\epsilon}{\|u\|}(x-\|u\|)|.$$
%	{Indeed, if the point $w$ is such that $|y|<|\frac{\epsilon}{\|u\|}(x-\|u\|)|$ then, since $u\in \partial G$, we would have $w\notin \partial G $.}
{Now,} suppose that $\cos(u,w)=\frac{x}{\|w\|}=\theta>0$. We have $y^2=\frac{1-\theta^2}{\theta^2}x^2$ and hence, by our claim,
$$\textstyle (x-\|u\|)^2\leq\frac{1-\theta^2}{\theta^2}\frac{\|u\|^2}{\epsilon^2} x^2.$$
Hence,
$$\textstyle \|u-w\|^2=(x-\|u\|)^2+y^2\leq x^2(\frac{\|u\|^2}{\epsilon^2}+1)\frac{1-\theta^2}{\theta^2}\leq K^2(\frac{K^2}{\epsilon^2}+1)\frac{1-\theta^2}{\theta^2}.$$
\end{proof}

\begin{proposition}\label{prop:eventuallycos<1}
 Let Basic assumptions~\ref{ba} be satisfied and, for each $n\in\N$, let $a_n\in A_n$ and $b_n\in B_n$. Suppose that the couple $(A,B)$ is regular.  
Let $\epsilon>0$, then there exist $\eta\in(0,1)$ and $n_1\in\N$ such that for each $n\geq n_1$ we have:
\begin{enumerate}
	\item if $\mathrm{dist}(a_n,E)\geq2\epsilon$	and $\mathrm{dist}(b_n,F)\geq2\epsilon$ then $\cos\bigl(a_n-e,b_n-(e+v)\bigr)\leq\eta,$ whenever  $e\in E+\epsilon B_X$.
\item if $\mathrm{dist}(a_n,E)\geq2\epsilon$	and $\mathrm{dist}(b_{n+1},F)\geq2\epsilon$ then $\cos\bigl(b_{n+1}-f,a_n+v-f\bigr)\leq\eta,$ whenever  $f\in F+\epsilon B_X$.
\end{enumerate}
\end{proposition}

\begin{proof} Let us prove that there exist $\eta\in(0,1)$ such that eventually (i) holds, the proof  that there exist $\eta\in(0,1)$ such that eventually (ii) holds is similar.  
	Suppose that this is not the case, then  there exist sequences $\{e_k\}\subset E+\epsilon B_X$, $ \{\theta_k\}\subset (0,1)$ and an increasing sequence of the integers $\{n_k\}$  such that $\mathrm{dist}(a_{n_k},E)\geq2\epsilon$, $\mathrm{dist}(b_{n_k},F)\geq2\epsilon$, and
	$$\cos\bigl(a_{n_k}-e_k,b_{n_k}-(e_k+v)\bigr)=\theta_k\to 1.$$
	Let $G=E+2\epsilon B_X$ and observe that $G$ is a bounded body in $X$. Since $e_k\in\inte G$ and $a_{n_k}\not\in\inte G$, there exists a unique point $a'_k\in[e_k,a_{n_k}]\cap\partial G$. Similarly, there exists a unique point $b'_k\in[e_k,b_{n_k}-v]\cap\partial G$.
	Moreover, it is clear that $$\cos\bigl(a'_{k}-e_k,b'_{k}-e_k\bigr)=\theta_k.$$
	Lemma~\ref{lemma:2dimensional} implies that $\|a'_{k}-b'_{k}\|\to0$. 
	 Since $G$ is bounded and $A_{n_k}\to A,\,B_{n_k}\to B$ for the Attouch-Wets convergence, there exist sequences $\{a''_k\}\subset A$ and $\{b''_k\}\subset B-v$ such that $\|a''_k-a'_k\|\to 0$ and $\|b''_k-b'_k\|\to 0$. Hence, $\|a''_{k}-b''_{k}\|\to0$ and eventually $\mathrm{dist}(a''_k, E)\geq \epsilon$, a contradiction  since the couple $(A,B)$ is regular.
\end{proof}

\begin{lemma}\label{prop: eventuallycos0}
	 Let Basic assumptions~\ref{ba} be satisfied, suppose that the couple $(A,B)$ is regular, and let $\delta,\epsilon>0$. For each $n\in\N$, let $a_n,x_n\in A_n$ and $b_n,y_n\in B_n$ be  such that $\mathrm{dist}(x_n, E)\to 0$ and $\mathrm{dist}(y_n, F)\to 0$. 
	Then there exists  $n_2\in\N$ such that for each $n\geq n_2$ we have:
	\begin{enumerate}
		\item if  $\mathrm{dist}(a_n,E)\geq2\epsilon$, $\mathrm{dist}(b_n,F)\geq2\epsilon$,  and $a_n=P_{A_n}b_n$ then 
$$\cos\bigl(x_n-a_n,b_n-(a_n+v)\bigr)\leq\delta;$$	
		\item if  $\mathrm{dist}(a_n,E)\geq2\epsilon$, $\mathrm{dist}(b_{n+1},F)\geq2\epsilon$, and $b_{n+1}=P_{B_{n+1}}a_n$ then 
$$\cos\bigl(y_{n+1}-b_{n+1},a_n+v-b_{n+1}\bigr)\leq\delta.$$	

\end{enumerate}
	\end{lemma}

\begin{proof} Let us prove that eventually (i) holds, the proof that eventually (ii) holds is similar. {Since $\mathrm{dist}(a_n,E)\geq 2 \varepsilon$ and $\mathrm{dist}(x_n,E)\rightarrow 0$ we have that eventually $x_n-a_n\neq0$}. By Proposition~\ref{prop:eventuallycos<1},  there exists $\eta\in(0,1)$ and $n_1\in\N$ such that
	$$\cos\bigl(a_n-e,b_n-(e+v)\bigr)\leq\eta,$$
	whenever $n\geq n_1$ and $e\in E+\epsilon B_X$. Since $\mathrm{dist}(a_n,E)\geq2\epsilon$	and $\mathrm{dist}(b_n,F)\geq2\epsilon$, it is not difficult to see  that there exists a constant $\eta'>0$ such that $\|b_n-(a_n+v)\|\geq\eta'$, whenever $n\geq n_1$. In particular, eventually $\cos\bigl(x_n-a_n,b_n-(a_n+v)\bigr)$ is well-defined. 
	    If $v=0$, the thesis is trivial since 
	$$\langle x_n-a_n,b_n-a_n\rangle\leq0,$$
	whenever $n\in\N$. 
	
	Suppose that $v\neq 0$. 
		We claim that, if $v$ denotes the displacement vector for the couple $(A,B)$,  eventually we have $$\textstyle \langle v,a_n-x_n\rangle\leq \delta\eta' \|a_n-x_n\|.$$
	To prove our claim observe that, since $\mathrm{dist}(x_n, E)\to 0$, we can suppose without any loss of generality that $\mathrm{dist}(x_n, E)\leq\epsilon$ ($n\in\N$). Moreover, we can consider  a sequence $\{x'_n\}\subset E$ such that $\|x'_n-x_n\|\to0$. Let $G=E+2\epsilon B_X$ and observe that $G$ is a bounded body in $X$. Since $x_n\in\inte G$ and $a_{n}\not\in\inte G$, there exists a unique point $a'_n\in[x_n,a_{n}]\cap\partial G$. 
	Since $G$ is bounded and $A_{n}\to A$ for the Attouch-Wets convergence, there exists a sequence $\{a''_n\}\subset A$  such that $\|a''_n-a'_n\|\to 0$. Since $\{x'_n\}\subset E$, it is clear that 
	$\langle v, a''_n-x'_n\rangle\leq 0$ and hence eventually
	$$\langle v, a''_n-x'_n\rangle-\delta\eta'\|a''_n-x'_n\|\leq -\delta\eta'\epsilon.$$	
	Since $\|x'_n-x_n\|\to0$ and $\|a''_n-a'_n\|\to0$, eventually we have
	$$\langle v, a'_n-x_n\rangle-\delta\eta'\|a'_n-x_n\|\leq0.$$
	By homogeneity and by our construction the claim is proved.
	
	Now, by our claim, since $a_n=P_{A_n}b_n$ and $x_n\in A_n$ ($n\in\N$), we have  
	 $$\textstyle \langle x_n-a_n,b_n-(a_n+v)\rangle=\langle x_n-a_n,b_n-a_n\rangle+\langle a_n-x_n,v\rangle\leq \delta\eta' \|a_n-x_n\|.$$
	Eventually,  since $\|b_n-(a_n+v)\|\geq\eta'$,   we have   
	 	$$\cos\bigl(x_n-a_n,b_n-(a_n+v)\bigr)\leq \frac{\delta\eta'}{\|b_n-(a_n+v)\|}\leq    \delta.$$
	
\end{proof}
{
\noindent Now, we need a simple geometrical result whose proof is a simple application of the law of cosines combined with the triangle inequality. The details of the proof are left to the reader.} 

\begin{fact}\label{fact: quasiortho}
Let  $\eta,\eta'\in(0,1)$ be such that $\eta<\eta'$. If $\delta\in(0,1)$ satisfies $\frac{\delta+\eta}{1-\delta}\leq \eta'$ and if $x,y\in X$ are linearly independent vectors  such that $\cos(x,y)\leq \eta$ and $\cos(y-x,-x)\leq \delta$ then $\|x\|\leq \eta'\|y\|$. 
\end{fact}

Let us recall that, given a normed space $Z$,  the {\em modulus of convexity of $Z$} is the   function $\delta_Z:[0,2]\to [0,1]$ defined by 
$$
\delta_Z(\eta)=\inf \left\lbrace 1-\left\| \dfrac{x+y}{2} \right\|:x,y \in B_X, \|x-y\|\geq \eta\right\rbrace. 
$$
Moreover, we say that $Z$ is  {\em uniformly rotund} if $\delta_Z(\eta)>0$, whenever $\eta \in (0,2] $.

\begin{lemma}\label{lemma: unifrotund} 
	Let $Z$ be a uniformly rotund {normed} space. For each $\rho\geq 0$ and $M>0$ there exists $\epsilon'>0$ such that if  $C$ is a convex set such that $\rho-\epsilon'\leq\|c\|\leq\rho+\epsilon'$, whenever $c\in C$, then $\mathrm{diam}(C)\leq M$.
\end{lemma}

\begin{proof}
In the case $\rho= 0$ the proof is trivial, so we can suppose   $\rho>0$. We claim that if we take  $\epsilon'>0$ such that $\varepsilon'\left(2-\delta_Z(M) \right) < \rho \delta_Z(M)$ the thesis follows. To see this,  suppose that $C$ is a convex set such that  $\rho-\epsilon'\leq\|c\|\leq\rho+\epsilon'$, whenever $c\in C$, and suppose on the contrary that there exist $c_1,c_2 \in C$ such that $\|c_1-c_2\|> M$.  By the definition of $\delta_Z$ and since $\frac{c_1+c_2}{2} \in C$, we have 
	$$
\rho - \varepsilon'\leq\left\| \dfrac{c_1+c_2}{2}\right\| \leq \left(  \rho + \varepsilon'\right) \left( 1- \delta_Z(M)\right).  
$$
 Therefore, we have $\varepsilon'\left(2-\delta_Z(M) \right) \geq \rho \delta_Z(M)$, a contradiction.  
\end{proof}

{Since it is well known that a Hilbert space is a uniformly rotund space, the previous lemma allows us to prove the following proposition. }

\begin{proposition}\label{prop: normprojectionsmall} Let  Basic assumptions~\ref{ba} be satisfied. For each $M>0$ there exist $\theta\in(0,M)$ and $n_0\in\N$ such that if $n\geq n_0$ we have:
	\begin{enumerate}
		\item if   $b_n\in B_n$, $a_n=P_{A_n}b_n$,  and $\mathrm{dist}(b_n, F)\leq \theta$ then $$\mathrm{dist}(a_n, E)\leq 2M;$$
		\item if  $a_n\in A_n$, $b_{n+1}=P_{B_{n+1}}a_n$,  and $\mathrm{dist}(a_n, E)\leq \theta$ then $$\mathrm{dist}(b_{n+1}, F)\leq 2M.$$
\end{enumerate} 
\end{proposition}

\begin{proof} Let $M>0$ and $\rho=\|v\|$, {where $v$ is the displacement vector}. 
 Let $\epsilon'\in(0,3M)$ be given by Lemma~\ref{lemma: unifrotund}. Put $\theta=\epsilon'/3$, since Basic assumption~\ref{ba} are satisfied, there exists $n_0\in \N$ such that if $n\geq n_0$ we have:
	\begin{enumerate}
		\item[(a)] if $w\in A_n$ then $\mathrm{dist}(w,F)\geq \rho-3\theta$;
		\item[(b)] if $e\in E$, there exists $x\in A_n$ such that  $\|e-x\|\leq\theta$.
	\end{enumerate} 
Now, let $n\geq n_0$, $b_n\in B_n$, $a_n=P_{A_n}b_n$,  and $\mathrm{dist}(b_n, F)\leq\theta$. Let $f_n\in F$ be such that $\|f_n-b_n\|\leq\theta$ and put $e_n=f_n-v\in E$. By (b), there exists $x_n\in A_n$ such that $\|x_n-e_n\|\leq \theta$. Hence, since $a_n=P_{A_n}b_n$ 
and $\|e_n-f_n\|=\rho$, we have
\begin{eqnarray*}
 \|a_n-f_n\|&\leq& \|a_n-b_n\|+\|f_n-b_n\|\\
&\leq& \|x_n-b_n\|+\|f_n-b_n\|\\
&\leq& \|x_n-e_n\|+ \rho +2\|f_n-b_n\|\leq\rho+3\theta.
\end{eqnarray*}  
	Let us consider the convex set $C=[x_n-f_n,a_n-f_n]$. Observe that, since $$\|x_n-f_n\|\leq\|e_n-x_n\|+\|e_n-f_n\|\leq\rho+\theta,$$ we have that 
	$\|c\|\leq\rho+3\theta$, whenever $c\in C$.
	Moreover, since $C\subset A_n$ and $f_n\in F$, by (a) we have $\|c\|\geq\rho-3\theta$, whenever $c\in C$. Hence, we
	 can apply Lemma~\ref{lemma: unifrotund} to the set $C$ and we have $\|a_n-x_n\|=\mathrm{diam}(C)\leq M$. Then
	$$\mathrm{dist}(a_n, E)\leq \|a_n-e_n\|\leq \|a_n-x_n\|+\|e_n-x_n\|\leq M+\theta\leq 2M.$$
		The proof that  eventually (ii) holds is similar. 
\end{proof}

{
We are now ready to state and prove the main result of this paper.}

\begin{theorem}\label{teo: mainHilbert}
	 Let $A,B$ be   closed convex nonempty subsets of $X$ such that $E$ and $F$ are bounded. Suppose that the couple $(A,B)$ is regular, then the couple $(A,B)$ is $d$-stable.
	
\end{theorem}

\begin{proof}
Let $a_0\in X$ and let $\{a_n\}$ and $\{b_n\}$ {be} the corresponding perturbed alternating projections sequences, { i.e,
$$a_n=P_{A_n}(b_n) \quad \text{and} \quad b_n=P_{B_n}(a_{n-1}).$$ 
First of all, we remark that it is enough to prove that $\mathrm{dist}(a_n, E)\to 0$ since the proof that $\mathrm{dist}(b_n, F)\to 0$ follows by the symmetry of the problem. 
Therefore our aim is to prove that for each $M>0$,  eventually we have $$\mathrm{dist}(a_n, E)\leq M.$$}

By applying Proposition~\ref{prop: normprojectionsmall} twice, there exists $0<\epsilon<M/2$ and $n_0\in\N$ such that, for each $n\geq n_0$, we have:
	\begin{enumerate}
	\item[($\alpha_1$)] if $\mathrm{dist}(b_n, F)\leq 2\epsilon$ then $\mathrm{dist}(a_n, E)\leq M$ and $\mathrm{dist}(b_{n+1}, F)\leq M$;\\
	 \item[($\alpha_2$)] if  $\mathrm{dist}(a_n, E)\leq 2\epsilon$ then $\mathrm{dist}(b_{n+1}, F)\leq M$ and $\mathrm{dist}(a_{n+1}, E)\leq M$.
\end{enumerate}

{Now, }by Proposition~\ref{prop:eventuallycos<1} there exist $\eta\in(0,1)$ and $n_1\geq n_0$ such that: 
\begin{enumerate}
	\item {if $\mathrm{dist}(a_n,E)\geq2\epsilon$	and $\mathrm{dist}(b_n,F)\geq2\epsilon$ then $$\cos\bigl(a_n-e,b_n-(e+v)\bigr)\leq\eta,$$ 
	whenever $n\geq n_1$ and $e\in E+\epsilon B_X$;}	
	\item {if $\mathrm{dist}(a_n,E)\geq2\epsilon$	and $\mathrm{dist}(b_{n+1},F)\geq2\epsilon$ then $$\cos\bigl(b_{n+1}-f,a_n-(f-v)\bigr)\leq\eta,$$ 
	whenever $n\geq n_1$ and $f\in F+\epsilon B_X$.}
\end{enumerate}

For each $n\in\N$, let  $e_n\in E$ and $f_n\in F$ be such that $\|a_n-e_n\|=\mathrm{dist}(a_n,E)$ and $\|b_n-f_n\|=\mathrm{dist}(b_n,F)$.
Let us consider  sequences $\{x_n\}$ and $\{y_n\}$ such that $x_n\in A_n, y_n\in B_n$ ($n\in\N$) and such that $\|x_n+v-f_n\|\to0$ and
$\|y_{n+1}-v-e_n\|\to0$. Moreover, without any loss of generality we can suppose that $x_n\in E+\epsilon B_X$ and $y_n\in F+\epsilon B_X$, whenever $n\geq n_1$.
  Now, take $\eta'\in(\eta,1)$ and  $\delta\in(0,1)$ satisfying $\frac{\delta+\eta}{1-\delta}\leq \eta'$.
By applying Proposition~\ref{prop: eventuallycos0} there exists $n_2\geq n_1$  such that,
for each $n\geq n_2$, we have:
\begin{enumerate}
	\item[(iii)] {if $\mathrm{dist}(a_n,E)\geq2\epsilon$	and $\mathrm{dist}(b_n,F)\geq2\epsilon$}    
	then 
	$$\cos\bigl(x_n-a_n,b_n-(a_n+v)\bigr)\leq\delta;$$	
	\item[(iv)] if {$\mathrm{dist}(a_n,E)\geq2\epsilon$	and $\mathrm{dist}(b_{n+1},F)\geq2\epsilon$}
	then 
	$$\cos\bigl(y_{n+1}-b_{n+1},a_n-(b_{n+1}-v)\bigr)\leq\delta.$$	
\end{enumerate}

Taking into account (i)-(iv) and Fact~\ref{fact: quasiortho}, if $n\geq n_2$ then the following conditions hold:
\begin{itemize}
	\item if $\mathrm{dist}(a_n,E)\geq2\epsilon$	and $\mathrm{dist}(b_n,F)\geq2\epsilon$ then {$$ \|a_n-x_n\|\leq\eta'\|b_n-(x_n+v)\|;$$}
	\item if $\mathrm{dist}(a_n,E)\geq2\epsilon$	and $\mathrm{dist}(b_{n+1},F)\geq2\epsilon$ then  {$$\|b_{n+1}-y_{n+1}\|\leq\eta'\|a_n-(y_{n+1}-v)\|.$$}
\end{itemize}

{Now}, by the triangle inequality, we have that, for each
 $n\geq n_2$,  the following conditions hold:
\begin{itemize}
	\item if $\mathrm{dist}(a_n,E)\geq2\epsilon$	and $\mathrm{dist}(b_n,F)\geq2\epsilon$ then {$$ \mathrm{dist}(a_n,E)\leq\|a_n-(f_n-v)\|\leq\eta'(\|b_n-f_n\|+\|x_n+v-f_n\|)+\|x_n+v-f_n\|;$$
	}
	\item if $\mathrm{dist}(a_n,E)\geq2\epsilon$	and $\mathrm{dist}(b_{n+1},F)\geq2\epsilon$ then 
	\begin{eqnarray*}
	 \mathrm{dist}(b_{n+1},F)&\leq&\|b_{n+1}-(e_n+v)\|\\
	 &\leq&\eta'(\|a_n-e_n\|+\|y_{n+1}-v-e_n\|)+\|y_{n+1}-v-e_n\|.
	\end{eqnarray*}
\end{itemize}

If we consider $\eta''\in(\eta',1)$, since {$\|x_n+v-f_n\|\to0$} and {$\|y_{n+1}-v-e_n\|\to0$}, there exists $n_3\geq n_2$ such that if 
$n\geq n_3$ then the following conditions hold:
\begin{enumerate}
	\item[($\beta_1$)] if $\mathrm{dist}(a_n,E)\geq2\epsilon$	and $\mathrm{dist}(b_n,F)\geq2\epsilon$ then $$ \mathrm{dist}(a_n,E)\leq\eta''\mathrm{dist}(b_n,F);$$
	 \item[($\beta_2$)]if $\mathrm{dist}(a_n,E)\geq2\epsilon$	and $\mathrm{dist}(b_{n+1},F)\geq2\epsilon$ then $$ \mathrm{dist}(b_{n+1},F)\leq\eta''\mathrm{dist}(a_n,E).$$
\end{enumerate}

Now, it is easy to see that there exists $n_4\geq n_3$ such that $\mathrm{dist}(a_{n_4},E)\leq2\epsilon$	or $\mathrm{dist}(b_{n_4},F)\leq2\epsilon$. Indeed, since $\eta''<1$, the fact that $$\mathrm{dist}(a_n,E)\geq2\epsilon\ \ \text{and}\ \ \mathrm{dist}(b_n,F)\geq2\epsilon,\ \ \ \text{whenever}\ n\geq n_3,$$ contradicts the fact that { ($\beta_1$) and ($\beta_2$) are} satisfied whenever $n\geq n_3$. Since {($\alpha_1$), ($\alpha_2$) and ($\beta_1$), ($\beta_2$)} are satisfied whenever $n\geq n_3$, we obtain that $\mathrm{dist}(a_{n},E)\leq M$
whenever $n>n_4$.
\end{proof}

{If the intersection of $A$ and $B$ is nonempty, we obtain, as an immediate consequence of Theorem \ref{teo: mainHilbert}, the following result.}
\begin{corollary} 
	 Let $A,B$ be   closed convex nonempty subsets of $X$ such that $A\cap B$ is bounded and nonempty. If the couple $(A,B)$ is  regular then the perturbed alternating projections sequences $\{a_n\}$ and $\{b_n\}$ satisfy   $\mathrm{dist}(a_n,A\cap B)\to 0$ and $\mathrm{dist}(b_n,A\cap B)\to 0$
	\end{corollary}

{We conclude this section by putting in evidence some relationships between the results of \cite{DebeMigl} and Theorem \ref{teo: mainHilbert}.}
	{First of all, we briefly recall some notions.}	
{	\begin{definition}[{see, e.g., \cite[Definition~7.10]{FHHMZ}}]\label{def:strexp} Let $A$ be a nonempty subset of a normed space $Z$. A point $a\in A$
		is called a strongly exposed point of $A$ if there exists  a
		support functional $f\in Z^*\setminus\{0\}$ for $A$ at $a$ $\bigl($i.e.,
		$f (a) = \sup f(A)$$\bigr)$, such that  $x_n\to a$ for all sequences
		$\{x_n\}$ in $A$ such that $\lim_n f(x_n) = \sup f(A)$. In this
		case, we say that $f$ strongly exposes $A$ at $a$.
	\end{definition}}

	\begin{definition}[{see, e.g., \cite[Definition~1.3]{KVZ}}]
		Let $A$ be a body in a normed space $Z$. We say that $x\in\partial A$ is an
		{\em LUR (locally uniformly rotund) point} of $A$ if for each
		$\epsilon>0$ there exists $\delta>0$ such that if $y\in A$
		and $\dist(\partial A,(x+y)/2)<\delta$ then $\|x-y\|<\epsilon$. 
	\end{definition}
	We say that $A$ is an {\em LUR body} if each point in
	$\partial A$ is an LUR point of $A$. The following lemma shows that each LUR point is a strongly exposed point. 	
	\begin{lemma}[{\cite[Lemma~4.3]{DebeMiglMol}} ]\label{slicelimitatoselur} Let $A$ be a body in a normed space $Z$
		and suppose that $a\in\partial A$ is an LUR point of $A$. Then, if
		$f\in S_{Z^*}$ is a support functional for $A$ in $a$, $f$
		strongly exposes $A$ at $a$.
	\end{lemma} 

First, we show that a more general variant of the assumptions of one of the main results in \cite{DebeMigl}, namely \cite[Theorem~3.3]{DebeMigl}, imply that the couple $(A,B)$ is regular. It is interesting to remark that here we consider also the case in which $A$ and $B$  do not intersect.
		 	
	\begin{proposition} \label{prop:stronglyexpregular}
		Let $A,B$ be nonempty closed convex subsets of $X$. Let us suppose that there exist $e \in A \cap (B-v)$ and a linear continuous functional $x^*\in S_{X^*}$ such that
		$$ \inf x^*(B-v)=x^*(e)=\sup x^*(A)$$
		and such that $x^*$ strongly exposes $A$ at $e$. Then the couple $(A,B)$ is regular. 
	\end{proposition}
\begin{proof} {There is no loss of  generality in assuming $e=0$.}
	It is a simple matter to see that $E=\{0\}$. Now, suppose on the contrary that  that $(A,B)$ is not regular. Therefore there exist  sequences $\{x_n\}\subset X$, $\{a_n\}\subset A$, $\{b_n\}\subset B$, and a real number $\bar{\varepsilon}>0$ such that 
	\begin{equation} \label{dist e}
		\mathrm{dist}(x_n,E)=\|x_n\|>\bar{\varepsilon}, 
	\end{equation}
		and such that
\begin{equation}\label{eq: notregular}
\mathrm{dist}(x_n,A)=\|x_n -a_n\|\rightarrow 0, \quad \mathrm{dist}(x_n,B-v)=\|x_n -b_n+v\|\rightarrow 0.
\end{equation}	
	
		 By (\ref{eq: notregular}) and since $ {\inf x^*(B-v)}=0=\sup x^*(A)$, it holds $\lim_n x^*(x_n)=0$ and hence $\lim_n x^*(a_n)=0$. Since $x^*$ strongly exposes $A$ at $e$, the last equality implies that $\|a_n\|\rightarrow 0$. We conclude that $\|x_n\|\rightarrow 0$, contrary to (\ref{dist e}).
\end{proof}

By combining the previous proposition and Theorem~\ref{teo: mainHilbert}, we obtain the following corollary generalizing  \cite[Theorem~3.3]{DebeMigl}.

\begin{corollary}\label{Corollary:stronglyexp}
	Let  $A,B$ be
	nonempty closed convex subsets of $X$.  
	Let us suppose that there exist $e \in A \cap (B-v)$ and a linear continuous functional $x^*\in S_{X^*}$ such that
	$$ \inf x^*(B-v)=x^*(e)=\sup x^*(A)$$
	and such that $x^*$ strongly exposes $A$ at $e$.
	Then, 
	the couple $(A,B)$ is stable.
\end{corollary}

 Moreover, in \cite{DebeMigl}, the authors proved the following sufficient condition for the stability of a  couple $(A,B)$.

\begin{theorem}[{\cite[Theorem~4.2]{DebeMigl}}]\label{theorem:corpilur} Let $X$ be a Hilbert space and $A,B$
	nonempty closed convex subsets of $X$. 
	Suppose that $\inte(A\cap B)\neq\emptyset$, then the couple $(A,B)$ is stable.
\end{theorem}

By combining Corollary~\ref{Corollary:stronglyexp} and Theorem~\ref{theorem:corpilur},  we obtain the following sufficient condition for the stability of the couple $(A,B)$ generalizing \cite[Corollary~4.3, (ii)]{DebeMigl}.
	
\begin{corollary}\label{corollary:corpilur} Let $X$ be a Hilbert space, suppose that 	$A,B$ are bodies in $X$ and that $A$ is LUR. Then the couple $(A,B)$ is stable.
\end{corollary}

\begin{proof}
	If $\inte(A\cap B)\neq\emptyset$, {the thesis follows by applying  Theorem~\ref{theorem:corpilur}.} If $\inte(A\cap B)=\emptyset$,
	since $A$ and $B$ are bodies, we have $\mathrm{int}(A)\cap B=\emptyset$. Since $A$ is LUR the intersection $A \cap (B-v)$ reduces to a singleton $\{e\}$. By the Hahn-Banach theorem, there exists a linear functional $x^*\in X^*$ such that $$ \inf x^*(B-v)=x^*(e)=\sup x^*(A).$$ Since $A$ is an LUR body, by Lemma \ref{slicelimitatoselur}, we have that $x^*$ strongly exposes $A$ at $e$. We are now in position to apply Corollary \ref{Corollary:stronglyexp} and conclude the proof.
	\end{proof}

	Finally, we show that  \cite[Theorem~5.2]{DebeMigl}, follows by Theorem~\ref{teo: mainHilbert}.
\begin{corollary}\label{cor:sottospazisommachiusa}
	Let  $U,V$ be closed subspaces of $X$ such that $U\cap V=\{0\}$ and $U+V$ is closed. Then the couple $(U,V)$ is stable.
\end{corollary} 
\begin{proof}
	Since $U+V$ is closed, by \cite[Corollary~4.5]{BauschkeBorwein93}, the couple $(U,V)$ is regular. By Theorem~\ref{teo: mainHilbert}, the couple $(U,V)$ is $d$-stable. Since $U\cap V$ is a singleton, the couple $(U,V)$ is stable. 
\end{proof}

\section{Final remarks, examples, and an open problem}

Known examples show that the hypothesis about regularity of the couple $(A,B)$, in Theorem~\ref{teo: mainHilbert}, is necessary. To see this, it is indeed sufficient to consider any couple $(A,B)$ of sets such that $A\cap B$ is a singleton and such that, for a suitable starting point, the method of alternating projections does not converge (see \cite{Hundal} for such a couple of sets). 

A natural question is whether, in the same theorem, the hypothesis about regularity of the couple $(A,B)$ can be replaced by the weaker hypothesis  that ``for any starting point the method of alternating projections converges''. The answer to previous question is negative; indeed, in \cite[Theorem~5.7]{DebeMigl}, the authors provided an example of a couple $(A,B)$ of closed subspaces of a Hilbert space such that $A\cap B=\{0\}$ and such that the couple $(A,B)$ is not stable (and hence not $d$-stable since $A\cap B$ is a singleton).{It is interesting to observe} that, by the classical Von Neumann result \cite{vonNeumann}, the method of alternating projections converges for this couple of sets.

The next example shows that, if we consider  closed convex sets $A,B\subset X$    such that $A\cap B$ is nonempty and bounded, the regularity of the couple $(A,B)$ does not imply in general that  $(A,B)$ is stable. In particular, we cannot replace $d$-stability with stability in the statement of Theorem~\ref{teo: mainHilbert}.

\begin{example}[{\cite[Example~4.4]{DebeMigl}}] \label{ex: notconverge}
		Let $X=\R^2$ and let us consider the following compact convex subsets of $X$:
		\begin{eqnarray*}
			A&=&\textstyle \conv\{(1,1),(-1,1),(1,0),(-1,0)\};\\ 
					B&=&\textstyle\conv\{(1,-1),(-1,-1),(1,0),(-1,0)\}.
		\end{eqnarray*} 
	Then the couple $(A,B)$ is regular (see \cite[Theorem~3.9]{BauschkeBorwein93}) but not stable (see \cite[Example~4.4]{DebeMigl}). 
\end{example}  

{Now, the following example shows that, even in finite dimension, the hypothesis concerning the boundedness of the sets $E,F$ cannot be dropped in the statement of Theorem~\ref{teo: mainHilbert}. }   

\begin{example}\label{ex: EnotBounded}
Let $A,B$ be the  subsets of $\R^3$ defined by
$$A=\{(x,y,z)\in\R^3;\, z=0, y\geq 0\},\ \ \ B=\{(x,y,z)\in\R^3;\, z=0\},$$
then the following conditions hold:
\begin{enumerate}
	\item[(a)]  $A\cap B$ coincides with $A$ (and hence $(A,B)$ is regular); 
\item[(b)] $A\cap B$  is not bounded;
\item[(c)] the couple $(A,B)$ is not $d$-stable.
\end{enumerate}
\end{example}

The proof of (a) and (b) is trivial. To prove (c), we need the following lemma, whose elementary proof is left to the reader.

\begin{lemma}\label{lemma: example fin-dim} Let $A,B$ be defined as in Example~\ref{ex: EnotBounded}. For each $n\in\N$ and  $x_0\geq1$, let $P_{n,x_0}^1,P_{n,x_0}^2,P_{n,x_0}^3\in\R^3$  be defined by 
$$\textstyle
P_{n,x_0}^1=(x_0+n x_0,-1,0),\ \ 
P_{n,x_0}^2=(x_0+n x_0+\frac{1}{n x_0},0,0),\ \ 
P_{n,x_0}^3=(0,\frac1n,\frac1n).
$$	
Let $t_{n,x_0}$ be the line in $\R^3$ containing the points $P_{n,x_0}^1$ and $P_{n,x_0}^3$, and let $r_{n,x_0}$ be the ray    in $\R^3$ with initial point $P_{n,x_0}^1$ and containing the points  $P_{n,x_0}^2$. Let $A_{n,x_0},B_{n,x_0}$ be the closed convex subsets of $\R^3$ defined by
	$$A_{n,x_0}=\conv(t_{n,x_0}\cup r_{n,x_0}),\ \ \ B_{n,x_0}=\{(x,y,z)\in\R^3;\, z=0\}.$$
Then the following conditions hold.
\begin{enumerate}
	\item 
	for each $N\in\N$,
	$\textstyle \lim_n h_N(A_{n,x_0},A)= 0$, uniformly with respect to $x_0\geq1$;
	\item  For each $n\in\N$ and  $x_0\geq 1$, the alternating projections sequences, relative to the sets $A_{n,x_0},B_{n,x_0}$ and starting point $(x_0,0,0)$, converge to  $P^1_{n,x_0}$.  
\end{enumerate}	
	
\end{lemma}

\begin{proof}[Sketch of the proof of Example~\ref{ex: EnotBounded}, (c)] 
Fix the starting point $a_0=(1,0,0)$ and let $A_{1,1},B_{1,1}$, be defined by Lemma~\ref{lemma: example fin-dim}. Observe that, if we consider the points $a^1_k=(P_{A_{1,1}} P_{B_{1,1}})^k a_0$ ($k\in\N$), by Lemma~\ref{lemma: example fin-dim}, (ii), there exists $N_1\in\N$ such that $\mathrm{dist}(a^1_{N_1}, A\cap B)\geq\frac12$.
Define $A_n=A_{1,1}$ and $B_n=B_{1,1}=B$, whenever $1\leq n\leq N_1$. Then define $A_{N_1+1}=A$ and $B_{N_1+1}=B$, and observe that $a^2_0:=P_A P_B a^1_{N_1}=(x_1,0,0)$ for some $x_1\geq1$. Similarly,  if we consider the points $a^2_{k}=(P_{A_{2,x_1}} P_{B_{2,x_1}})^k a^2_{0}$ ($k\in\N$), then there exists $N_2\in\N$  such that $\mathrm{dist}(a^2_{N_2}, A\cap B)\geq\frac12$. 
Define $A_n=A_{2,x_1},B_n=B_{2,x_1}=B$, whenever $N_1+1< n\leq N_2$. Then define $A_{N_2+1}=A,B_{N_2+1}=B$, and observe that $a^3_{0}:=P_A P_B a_{N_2}^{2}=(x_2,0,0)$ for some $x_2\geq1$.

 Then, proceeding inductively, we can construct  sequences $\{A_n\}$ and $\{B_n\}$ such that, by Lemma~\ref{lemma: example fin-dim}, (i), $A_n\to A$ and $B_n\to B$ for the Attouch-Wets convergence.
  Moreover, by our construction,  
 it is easy to see that the corresponding perturbed alternating projections sequences  $\{a_n\}$ and $\{b_n\}$,  with starting point $a_0$, are such that 
 $$\textstyle \limsup_n \mathrm{dist}(a_n,A\cap B)\geq \frac12.$$
 This proves that the couple $(A,B)$ is not $d$-stable.
\end{proof}

%Another natural question is whether, in Theorem~\ref{teo: mainHilbert}, it is possible to avoid the condition 
%concerning the boundedness of the sets $E,F$, if we replace the Attouch-Wets convergence of the sequences $\{A_n\}$ and $\{B_n\}$ with the Hausdorff convergence of the same sequences. A careful  inspection of the proof of Theorem~\ref{teo: mainHilbert} shows that this is the case and hence the following theorem holds.
%
%\begin{theorem}\label{teo: mainHausdorff} Let $A,B$ be   closed convex nonempty subsets of a Hilbert space $X$ such that  the couple $(A,B)$ is regular.  Suppose that
%	$\{A_n\}$ and $\{B_n\}$ are sequences of closed convex sets such that $A_n\rightarrow
%	A$ and $B_n\rightarrow B$ for the Hausdorff convergence. 
%	Let
%	$a_0\in X$ and consider sequences  $\{a_n\}$ and $\{b_n\}$ defined  by
%	$$b_n=P_{B_n}(a_{n-1})\ \ \ \text{and}\ \ \  a_n=P_{A_n}(b_n) \ \ \ \ \ \ \ \ \ (n\in\N).$$	
%	Then 
%	$\mathrm{dist}(a_n, E)\to0$ and $\mathrm{dist}(b_n, F)\to0$.
%	
%\end{theorem} 

Finally, we conclude with an open problem asking whether the inverse of Theorem~\ref{teo: mainHilbert} holds true.

\begin{problem} 
Let $A,B$ be   closed convex nonempty subsets of $X$ such that $E$ and $F$ are nonempty and bounded. Suppose that the couple $(A,B)$ is $d$-stable. Does the couple $(A,B)$ is regular?	
\end{problem}

% Acknowledgments here
\section*{Acknowledgements.}
The research of the authors is partially
supported by GNAMPA-INdAM, Progetto GNAMPA 2020. {The second author is also partially supported by the Ministerio de Ciencia, Innovación y Universidades (MCIU), Agencia Estatal de Investigación (AEI) (Spain) and	Fondo Europeo de Desarrollo Regional (FEDER) under project PGC2018-096899-	B-I00 (MCIU/AEI/FEDER, UE)}
%% References here (outcomment the appropriate case)
%
%% CASE 1: BiBTeX used to constantly update the references
%%   (while the paper is being written).
%%\bibliographystyle{informs2014} % outcomment this and next line in Case 1
%\bibliographystyle{apacite}
%\bibliography{DeBernardiMiglierinaMolho} % if  more than one, comma separated
%
%% CASE 2: BiBTeX used to generate mypaper.bbl (to be further fine tuned)
%%\input{mypaper.bbl} % outcomment this line in Case 2
%

\end{document}